\documentclass[a4paper,11pt]{article}

\usepackage[a4paper]{geometry}
\usepackage{amsthm,amsmath,amsfonts,amssymb,mathrsfs}
\usepackage{graphicx,xcolor,color}
\usepackage[numbers]{natbib}
\usepackage[colorlinks,citecolor=blue,urlcolor=blue]{hyperref}

\theoremstyle{plain}
\newtheorem{theorem}{Theorem}[section]
\newtheorem{corollary}[theorem]{Corollary}

\newtheorem{lemma}[theorem]{Lemma}
\newtheorem{proposition}[theorem]{Proposition}

\theoremstyle{definition}

\theoremstyle{remark}
\newtheorem{remark}[theorem]{Remark}

\numberwithin{equation}{section}

\title{Random walks in time-inhomogeneous random environment conditioned to stay positive}
\author{Wenming Hong \thanks{School of Mathematical Sciences \& Laboratory of Mathematics and Complex Systems, Beijing Normal University, Beijing 100875, P.R. China. Email: wmhong@bnu.edu.cn} ~and Shengli Liang \thanks{School of Mathematical Sciences \& Laboratory of Mathematics and Complex Systems, Beijing Normal University, Beijing 100875, P.R. China. Email: liangshengli@mail.bnu.edu.cn}}

\begin{document}
	
	\maketitle	
	\begin{center}
		\begin{minipage}{12cm}
			\begin{center}\textbf{Abstract}\end{center}
			We consider a random walk $\{S_n\}_{n\in \mathbb{N}}$ in time-inhomogeneous random environment $\xi$. For almost each realization of $\xi$, we  formulate a quenched harmonic function, based on which  we can define the random walk in random environment conditioned to stay positive by the Doob's {\it h}-transform. 
			Furthermore,  we prove a quenched invariance principle for the  conditioned random walk for almost each realization of $\xi$.
		
			\bigskip
			
			\mbox{}\textbf{Keywords}: Random walk, time-inhomogeneous random environment, conditioned to stay positive,  quenched harmonic function, quenched invariance principles. \\
			\mbox{}\textbf{Mathematics Subject Classification}:  Primary 60G50;
			secondary 60K37, 60F17.
			
		\end{minipage}
	\end{center}


\section{Introduction and main results}

\subsection{Motivation}

Random walks conditioned to stay positive have received a substantial research attention. We refer to Tanaka \cite{Tan89}, Bertoin and Doney \cite{BD94}, Biggins \cite{Big03}, Caravenna \cite{Car05}, Vatutin and Wachtel \cite{VW09} to cite only a few. Recent progress has been made for more general case, Denisov and Wachtel \cite{DW15} considered multidimensional Random walks in general cones, Denisov, Sakhanenko and Wachtel \cite{DSW18} treated the random walk with independent but non-identically distributed	increments and Grama, Lauvergnat and Le Page \cite{GLL18} focused on Markov walks (which is time-homogeneous). On the other hand, applications of the random walks conditioned to stay positive appear in many situations. For instance, the asymptotic behaviour of the survival probability for critical branching processes in random environment is closely related to the behaviour of the associated random walk conditioned to stay positive, see Kozlov \cite{Koz76}, Geiger and Kersting \cite{GK00}, Afanasyev et al. \cite{AGKV05}. The convergence of minimal position, additive martingale and derivative martingale in branching random walks is connected to random walks conditioned to stay positive and harmonic function, see A{\"{\i}}d{\'e}kon \cite{Aid13}, A{\"{\i}}d{\'e}kon and Shi \cite{AS14}, Biggins and Kyprianou \cite{BK04}, Chen \cite{Che15}, Hu and Shi \cite{HS09}.

The phrase ``random walk conditioned to stay positive" has at least two different interpretations. In the first, we consider the random walk conditioned to stay positive up to time $n$, which is a discrete version of meander.  The second interpretation involves conditioning on the event that the random walk never enters $\left(-\infty,0\right]$, and so can be thought of as a discrete version of the Bessel process, i.e., random walks conditioned to stay positive is defined by Doob's {\it h}-transform with the help of a harmonic function.

\

The main purpose of this paper is to formulate random walks with random environment in time (abbreviated as RWRE) conditioned to stay positive on the interval $(0,\infty)$. To this end, a key step is to formulate a (quenched) harmonic function. However, under the quenched probability, for almost each realization of $\xi$, the RWRE $\{S_n\}_{n\in \mathbb{N}}$ is a sum of independent but not necessarily identically distributed random variables, one loses the duality property. Therefore, there is no hope to generalize the traditional approach, the Wiener-Hopf factorisation, to formulate the harmonic function by the underlying ladder process of the random walks. Inspired by the universality approach used in \cite{DSW18}, we prove the existence of the limit $\lim\limits_{n\to\infty}U_{n}(\xi,y)=U(\xi,y):=-\mathbb{E}_\xi(S_{\tau_y})$ for  almost each realization of $\xi$, where $U_n(\xi,y):=\mathbb{E}_\xi(y+S_n; \tau_{y}>n)$ and $\tau_{y}:=\inf\left\{n>0: y+S_n\leq0\right\}$. We find that as a function of the environment $\xi$, $U(\xi,y)$ is a harmonic function in the sense of $U(\xi,y)=\mathbb{E}_{\xi}\left[U(\theta \xi,y+S_{1});\tau_{y}>1\right]$, for $P\text{-a.s.}\xi$, and some  properties of the quenched harmonic function $U(\xi,y)$ have been derived (see Theorem \ref{Th1} below), which is one of the novelty of this paper. It should be pointed out that the quenched harmonic function $U(\xi,y)$ is a function of the trajectory of a realization of $\xi:=(\xi_1,\xi_2, \cdots)$, it can not be obtained in the annealed probability even for the independent and identically distributed environment.

With the quenched harmonic function $U(\xi,y)$ in hand, we can formulate the RWRE conditioned to stay positive on the interval $(0,\infty)$, whose law is defined by Doob's {\it h}-transform: for any $N\in \mathbb{N}^*$, $y\geq0$ and $B\in\sigma\left(\xi\right)\otimes\mathscr{F}^{\xi}_N$,
\begin{equation}\label{P+}
	\mathbb{P}^+_{\xi,~y}(B):=\frac{\mathbb{E}_{\xi,~y}\left(U(\theta^N \xi, S_N)1_{\{\tau_0>N\}}1_B\right)}{U(\xi,y)}.
\end{equation}
Furthermore, we will  prove a quenched functional central limit theorem for this conditioned walks.

In addition to the theoretical interests of their own, an original motivation to consider the RWRE conditioned to stay positive on the interval $(0,\infty)$, comes from the study of branching random walks with random environment in time (BRWRE). As usual, by the celebrated many-to-one lemma, one can relate BRWRE with RWRE, by analysing the behaviour of RWRE to solve some questions on BRWRE. For example, Mallein and Mi\l{}o{\'s} \cite{MM18,MM19} find the asymptotic behaviour of maximal position of BRWRE by studying the ballot theorem for the associated RWRE. Our attempt is to use the quenched harmonic function to construct the RWRE conditioned to stay positive and to investigate the convergence of derivative martingale and Seneta-Heyde type theorem for the BRWRE, which is ongoing work.



\subsection{Model}

Consider a model of random walk $S$ in a random environment $\xi$. An environment is given by a family $\xi=\{\xi_{n}\}_{n\in \mathbb{N}^*}$ of random variables defined on some probability space $\left(\Omega_{E},\mathscr{F}_{E},P\right)$, with values in the space of probability laws on $\mathbb{R}$ . In this paper, we assume that the environment $\{\xi_{n}\}_{n\in \mathbb{N}^*}$ are {\it stationary and ergodic}
random variables. Each realization of $\xi_{n}$ corresponds to a probability law $\eta_{n}$ on $\mathbb{R}$. When the environment $\xi=\{\xi_{n}\}_{n\in \mathbb{N}^*}$ is given, the process can be described as a time-inhomogeneous random walk as follows. Let $\{X_{n}\}_{n\in \mathbb{N}^*}$ be a sequence of independent real valued random variables such that $X_n$ has probability law  $\eta_n$. Set
$$S_{0}:=x\in\mathbb{R},~~~~ S_n:=S_0+\sum_{k=1}^{n}X_i,~~~~ n\geq1.$$
The process $\{S_n\}_{n\in \mathbb{N}}$ constructed above is a {\it random walk with random environment in time starting from $x$}. Let $\left(\Omega_{RW},\mathscr{F}_{RW},\mathbb{P}_{\xi}\right)$ be the probability space under which the random walk is defined when the environment $\xi$ is fixed, and $\mathscr{F}^{\xi}_{n}:=\sigma\{S_k;k\leq n\}$ be the natural filtration of the process $S$. The total probability space can be formulated as the product space $\left(\Omega_{E}\times\Omega_{RW},\mathscr{F}_{E}\otimes\mathscr{F}_{RW},\mathbb{P}\right)$, where $\mathbb{P}=P\otimes \mathbb{P}_{\xi}$ in the sense that for any non-negative measurable function $f$, we have
$$\int f\,d\mathbb{P}=\int_{\Omega_{E}}\left(\int_{\Omega_{RW}}f\left(\xi,\omega\right)d\mathbb{P}_{\xi}\right)dP.$$
Here, we call $\mathbb{P}_{\xi}$ {\it quenched probability} and $\mathbb{P}_{x}$ {\it annealed probability}. Note that quenched probability is the conditional probability of $\mathbb{P}$ given the environment $\xi$. The corresponding expectations with respect to $\mathbb{P}_{\xi}$ and $\mathbb{P}$ will be denoted by $\mathbb{E}_{\xi}$ and $\mathbb{E}$, respectively.

Observe that in this model, we sample the environment randomly on time-axis. In many other studies, the environment is sampled randomly in space, see  Zeitouni \cite{Zei04}, or space-time, cf. Rassoul-Agha and Sepp{\"a}l{\"a}inen \cite{RS05}. Note that in the quenched sense, our model is the sums of independent, not necessarily identically distributed random variables.



\subsection{Quenched harmonic function}

We shall always assume that
\begin{equation}\label{assumption}
	\begin{aligned}
		&\mathbb{E}_\xi(X_1)=0, ~~~~P\text{-a.s.};\\
		&\mathbb{E}(X^{2+\epsilon}_1)<\infty, ~~~~\text{for some}~~ \epsilon>0.
	\end{aligned}
\end{equation}
To get rid of the trivial case, we assume throughout this paper that $\mathbb{P}_\xi(X_1>0)>0, ~P\text{-a.s.}$.

For any $y\geq0$, denote by $\tau_{y}$ the first moment when $\{S_n\}$ enters the interval $(-\infty, -y]$: $$\tau_{y}:=\inf\left\{n>0: y+S_n\leq0\right\}.$$
Define $$U_n(\xi,y):=\mathbb{E}_\xi(y+S_n; \tau_{y}>n).$$ Let $\theta$ be the shift operator, i.e. $\theta\xi:=(\xi_2,\xi_3,\cdots)$. For $n\geq1$, $\theta^n\xi:=\theta(\theta^{n-1}\xi)$ with the convention that $\theta^0\xi:=\xi$.

Under the quenched probability, for almost each realization of $\xi$, the RWRE $\{S_n\}_{n\in \mathbb{N}}$ is a sum of independent but not necessarily identically distributed random variables, 
one loses the duality property. Therefore, we can not formulate the harmonic function by the underlying ladder process of the random walks. To get over this difficulty, firstly, followed the strong approximation (universality) approach in \cite{DSW18}, we will prove the existence of the limit $\lim\limits_{n\to\infty}U_{n}(\xi,y)=U(\xi,y):=-\mathbb{E}_\xi(S_{\tau_y})$ for almost each realization of $\xi$ in the following proposition. For completeness, we shall give the details of the proof for this proposition in next section and simplify the proof.

\begin{proposition}\label{Prop1}
	Assume that (\ref{assumption}) holds, $y\geq0$, then there exists a random variable $U(\xi,y)$ such that, for almost all $\xi$,
	$$\lim\limits_{n\to\infty}U_{n}(\xi,y)=U(\xi,y):=-\mathbb{E}_\xi(S_{\tau_y}).$$
\end{proposition}

Secondly, note that the limiting function $U(\xi,\cdot)$ is random and measurable with respect to $\sigma\left(\xi\right):=\sigma\left(\xi_{1},\xi_2,\cdots\right)$. It depends on the whole sequence $\xi=\{\xi_{n}\}_{n\in \mathbb{N}^*}$ of random variables. By shifting the environment, we obtain a family $\left\{U\left(\theta^n\xi,\cdot\right)\right\}_{n\in \mathbb{N}}$ of functions. Based on the observation that the function $U\left(\xi,\cdot\right)$ is shifted invariant for the semigroup of the random walk killed when it first enters $(-\infty,0]$ in the quenched sense, we will prove that $U\left(\xi,\cdot\right)$ can be served as a positive quenched harmonic function which plays an essential role in defining the quenched random walks conditioned to stay positive. The ``usual" properties of the harmonic function are established as well.
\begin{theorem}\label{Th1}
	Assume that (\ref{assumption}) holds, $U(\xi,y)$ is defined as in Proposition \ref{Prop1}, then	
	\begin{item}
		\item[1.]
		$U$ is $\sigma$$\left(\xi\right)$-measurable and satisfies the quenched harmonic property:  $$U(\xi,y)=\mathbb{E}_{\xi}\left[U(\theta \xi,y+S_{1});\tau_{y}>1\right],~~~~P\text{-a.s.}.$$
		\item[2.]
		$U(\xi,y)$ is positive, non-decreasing in $y$ and
		\begin{equation}\label{asymptotic behaviour of harmonic function 1}
			\lim\limits_{y\to\infty}\frac{U(\xi,y)}{y}=1,~~~~P\text{-a.s.}.
		\end{equation}
	\end{item}
\end{theorem}

\begin{corollary}\label{Coro1}
	Under the assumption of Theorem \ref{Th1}, for almost all $\xi$, we have
	\begin{item}
		\item[1.]
		$\left\{U(\theta^n\xi,y+S_n)1_{\left\{\tau_{y}>n\right\}}\right\}_{n\in \mathbb{N^*}}$ is a martingale with respect to filtration $\left(\sigma\left(\xi\right)\otimes\mathscr{F}^{\xi}_{n}\right)$ under $\mathbb{P}_\xi.$
		\item[2.]
		For any $y_n\geq0$ with $y_n\to \infty$ as $n\to \infty$, 
		\begin{equation}\label{asymptotic behaviour of harmonic function 2}
			\lim\limits_{n\to\infty}\frac{U(\theta^n\xi,y_n)}{y_n}=1.
		\end{equation}
		\item[3.]
		The  quenched probability of the random walks  to stay positive is given by
		\begin{equation}\label{asymptotic behaviour of RW stay positive}
			\mathbb{P}_\xi(\tau_y>n)\sim \frac{\sqrt{2}U(\xi,y)}{\sqrt{\pi n}\sigma}, ~~~~\text{as}~n\to\infty,
		\end{equation}where $\sigma^2:=\mathbb{E}(X^2_1)$.
	\end{item}
\end{corollary}

\begin{remark}\label{Rem1}
	
	\
	
	\begin{itemize}
		\item[1.]
		For the case of non-random environment, one important tool to analyse the asymptotic behaviour of random walks conditioned to stay positive is the Wiener-Hopf factorisation, does not require any moment conditions on the random walks. For any oscillating random walks, a widely used fact is that the renewal function $R(y)$ associated with ladder heights processes is harmonic, we refer to the standard book of Feller \cite{Fel71}. These methods essentially rely on the so-called duality principle which, unfortunately, failed in the quenched setting. To get over this difficulty, firstly,  we  prove the existence of the limit $\lim\limits_{n\to\infty}U_{n}(\xi,y)=U(\xi,y)$ followed the strong approximation  approach  in \cite{DSW18}. Then, based on the observation that the function $U\left(\xi,\cdot\right)$ is shifted invariant for the semigroup of the random walk killed when it first enters $(-\infty,0]$ in the quenched sense, we prove that $U\left(\xi,\cdot\right)$ can be served as a positive quenched harmonic function. This is one of the novelty of the present paper.

		\item[2.]
		The quenched harmonic function $U(\xi,y)$ plays a key role. Thanks to the claim 1 of Corollary \ref{Coro1}, we can define the  RWRE conditioned to stay positive on the interval $(0,\infty)$, whose law is given by Doob's {\it h}-transform as in (\ref{P+}).


		\item[3.]
		It is worth while to note that $U(\xi,y)$ depends on the whole sequence $\xi=\{\xi_{n}\}_{n\in \mathbb{N}^*}$ of random variables. Although $\left\{U(\theta^n\xi,y+S_n)1_{\left\{\tau_{y}>n\right\}}\right\}_{n\in \mathbb{N^*}}$ is a martingale with respect to filtration $\left(\sigma\left(\xi\right)\otimes\mathscr{F}^{\xi}_{n}\right)$ under $\mathbb{P}_\xi$, it is not a martingale with respect to the filtrations  $\left(\sigma\left\{\xi_1,\cdots,\xi_n\right\}\otimes\mathscr{F}^{\xi}_{n}\right)$ under the annealed law $\mathbb{P}$. The quenched harmonic function $U(\xi,y)$ can not be obtained in the annealed probability even for the independent and identically distributed environment. Some other interesting properties related the  quenched harmonic function $U(\xi,y)$ see section 2 of \cite{HL22}.


		\item[4.]
		For time-homogeneous Markov walks, Grama et al. \cite{GLL18} obtained a harmonic function in their Theorem 2.2. However, our Theorem \ref{Th1} is not included in their results, since our random walks is time-inhomogeneous given the environment. Indeed, in the proof of propositions \ref{Prop1}, \ref{Th2} and claim 3 of Corollary \ref{Coro1}, we are benefited from Denisov et al. \cite{DSW18}, where the time-inhomogeneous random walks has been investigated.
	\end{itemize}
\end{remark}

\subsection{Quenched invariance principle}


The invariance principles for random walks conditioned to stay positive has been treated by many authors.
Iglehart \cite{Igl74} and Bolthausen \cite{Bol76} proved invariance principles for random walks conditioned to stay positive over a finite time interval: the law of the rescaled process $S_{(n\cdot)}$ conditioned to stay positive on the interval $(0,n]$, converges weekly to the law of the Brownian motion conditioned to stay positive on the interval $(0,1]$ (the law of the Brownian meander). While Bryn-Jones and Doney \cite{BD06} consider the case when the random walk is conditioned to stay positive over the infinite interval $\left(0,\infty\right)$: the rescaled process $S_{(n\cdot)}$ conditioned to stay positive on the interval $(0,\infty)$ (defined with the help of the harmonic function), converges in law to the Bessel process, which is the Brownian motion conditioned to stay positive on the interval $(0,\infty)$ (defined with the help of the corresponding harmonic function). In the latter case, the harmonic function plays a very important role in constructing these conditioned processes. For the stable law case, Doney \cite{Don85} and Caravenna and Chaumont \cite{CC08} considered the corresponding invariance principles respectively.

Now we are interested to investigate the quenched invariance principles for random walks with random environment in time conditioned to stay positive. In the ``Brownian meander" case, the quenched invariance principles followed directly from the results of Denisov, Sakhanenko and Wachtel \cite{DSW18} for random walks with independent but non-identically distributed increments, by checking the conditions there for almost each realization of $\xi$, see Theorem \ref{Th2} below. Our attention is focused on the ``Bessel process" case, the Brownian motion conditioned to stay positive on the interval $(0,\infty)$ is defined with the help of the corresponding harmonic function. The method of proof adapts a continuity argument used in Caravenna and Chaumont \cite{CC08} (originally comes from Bolthausen \cite{Bol76}) for the case of non-random environment. Avoiding calculating the finite-dimensional distributions and proving tightness, we follow a straightforward technique by exploiting the absolute continuity relation between the process conditioned to stay positive and the corresponding meander process, see Section 3.

\

To formulate the invariance principle more clearly, we use the terminology of canonical processes. We denote by $\Omega:=C[0,\infty)$ the space  of real valued continuous functions on the half-line $\mathbb{R}_+:=[0,\infty)$ endowed with the uniform topology and by $W:=\{W_t\}_{t\in\mathbb{R}_+}$ the coordinate process defined on this space:
$$\Omega\ni \omega \longmapsto W_t(\omega):=\omega_{t}.$$
Let $\mathbf{P}_{x}$ be the law on $\Omega$ of a standard Brownian motion started at $x$, i.e. $\mathbf{P}_{x}(W_0=x)=1$ and for simplicity set $\mathbf{P}:=\mathbf{P}_{0}$. Let $\mathscr{F}_{t}:=\sigma\{W_t;s\leq t\}$ be the natural filtration of the process $W$. We introduce the space $\Omega_t:=C[0,t]$ as the restricted of $\Omega$ to the time interval $[0,t]$.

Given the environment $\xi$, with an abuse of notation, we denote by $\Omega_{RW}:=\mathbb{R}^{\mathbb{N}}$ the space of discrete trajectories and by $S:=\{S_n\}_{n\in \mathbb{N}}$ the coordinate process defined on this space:
$$\Omega_{RW}\ni \omega^{RW} \longmapsto S_n(\omega^{RW}):=\omega^{RW}_{n}.$$
Let $\mathbb{P}_{\xi,~y}$ be the law on $\Omega_{RW}$ of a quenched random walk started at $y$, i.e. $\mathbb{P}_{\xi,~y}(S_0=y)=1$ and set $\mathbb{P}_{\xi}:=\mathbb{P}_{\xi,~0}$. Note that under $\mathbb{P}_{\xi,~y}$ the variables $\{S_n-S_{n-1}\}_{n\in \mathbb{N}^*}$ are independent but not necessarily identically distributed.


\subsubsection{``Brownian meander" case: QIP for random walk conditioned to stay positive up to time $N$}

We denote by $\mathbf{P}^{(m)}$ the law on $\Omega_1$ of the Brownian meander, which can be defined more explicitly as the limiting distribution:
\begin{equation}\label{Prob of Brownian meander}
	\mathbf{P}^{(m)}(\cdot):=\lim_{x\downarrow0}\mathbf{P}_x(\cdot \mid\underline{W}_1 >0),
\end{equation} where $\underline{W}_t:=\inf_{0\leq s\leq t}W_s$, see Theorem 2.1 in Durrett, Iglehart and Miller \cite{DIM77}. This clarifies that the Brownian meander is a Brownian motion conditioned to stay positive on the interval (0,1].

Similarly, we denote by $\mathbb{P}_{\xi}^{(m),N}$ the law on $\Omega_{RW}$ of the random walk conditioned to stay positive up to time $N$:
\begin{equation}\label{Prob of RW stay positive up to N}
	\mathbb{P}_{\xi}^{(m),N}(\cdot):=\mathbb{P}_{\xi}(\cdot\mid \tau_0>N).
\end{equation}

Define the rescaling mapping $\phi_N:\Omega_{RW}\longrightarrow \Omega$ as $$\Omega_{RW}\ni \omega^{RW} \longmapsto \left(\phi_N(\omega^{RW})\right)(t):=\frac{\omega^{RW}_{(Nt)}}{\sqrt{N}\sigma}, ~~~~t\in[0,\infty),$$ where
$$
\omega^{RW}_{(u)}:=\begin{cases}
	\omega^{RW}_{(k)}, & u=k\in \mathbb{N};\\
	\text{linear interpolation between~}\omega^{RW}_{(k)} \text{and~} \omega^{RW}_{(k+1)} , & u\in[k,k+1].
\end{cases}
$$
We introduce the rescaled meander measure on $\Omega_1$ (here $\phi_N$ is to be understood as a mapping from $\Omega_{RW}$ to $\Omega_1$):
\begin{equation}\label{rescaled meander measure on Omega_1}
	\mathbf{P}_{\xi}^{(m),N}:=\mathbb{P}_{\xi}^{(m),N}\circ (\phi_N)^{-1}.
\end{equation}

We have the following quenched invariance principle for the meander.
\begin{proposition}\label{Th2}
		Assume that (\ref{assumption}) holds, then for almost all $\xi$, we have
		\begin{equation}\label{QIP for meander}
			\mathbf{P}_{\xi}^{(m),N}\Longrightarrow \mathbf{P}^{(m)},~~~~\text{as}~~ N\to \infty,
		\end{equation} where $\Longrightarrow$ means convergence in law.
\end{proposition}

\begin{proof}
	In our setting, noting that $\left\{B^2_n-B^2_{n-1}:=\mathbb{E}_{\xi}(X^2_n)\right\}_{n\in\mathbb{N^*}}$ is a sequence of stationary and ergodic random variables with respect to $P$, by the ergodic theorem we have $\lim\limits_{n\to\infty}\frac{B^2_n}{n}=\sigma^2$, $P$-a.s., and under our condition (\ref{assumption}), for almost each realization of $\xi$, the Lindeberg condition (9) in of Denisov, Sakhanenko and Wachtel \cite{DSW18} holds. Then by the same routine as the proof of Theorem 3 of \cite{DSW18}, we can prove Proposition \ref{Th2}.
\end{proof}

In particular, for $t=1$, we have the quenched central limit theorem corresponding to the quenched invariance principle for the meander.
\begin{corollary}
	Assume that (\ref{assumption}) holds, then for almost all $\xi$, we have
	\begin{equation}\nonumber
		\lim_{N\to\infty}\mathbb{P}_{\xi}\left(\frac{S_{N}}{\sqrt{N}\sigma}\leq u\bigg|\tau_{0}>N\right)=1-e^{-\frac{u^2}{2}},~~~~\text{for any}~~u\geq0.
	\end{equation}
\end{corollary}


\subsubsection{``Bessel process" case: QIP for random walk conditioned to stay positive on the interval $\left(0,\infty\right)$}

We denote by $\mathbf{P}^+_x$ the law on $\Omega$ of the Brownian motion started from $x>0$ and conditioned to stay positive on the time interval $(0,\infty$), which can be defined by the well-known Doob's {\it h}-transform. More precisely, $V(x):=x1_{\{x>0\}}$ is the harmonic function of the Brownian motion killed when it first passes into non-positive half-line, for any $t>0$, $x>0$ and $A\in\mathscr{F}_t$, let
\begin{equation}\label{Prob of Bessel process}
	\mathbf{P}^+_x(A):=\frac{\mathbf{E}_x\left(V(W_t)1_{\{\underline{W}_t >0\}}1_A\right)}{V(x)}.
\end{equation}
For the case that the Brownian motion started at zero, we can define $\mathbf{P}^{+}(\cdot):=\lim\limits_{x\downarrow0}\mathbf{P}^+_x(\cdot)$ (see \cite{CC08}) and furthermore the measures $\mathbf{P}^{+}(\cdot)$ and $\mathbf{P}^{(m)}$ are absolutely continuous with respect to each other (see \cite{Cha97}). In particular, for any $A\in\mathscr{F}_1$,
\begin{equation}\nonumber
	\mathbf{P}^{+}(A)=\mathbf{E}^{(m)}\left(\sqrt{\frac{2}{\pi}}V(W_1)1_A\right).
\end{equation}

We introduce the law on $\Omega_{RW}$ of the quenched random walk started from $y\geq0$ and conditioned to stay positive on the time interval $(0,\infty$), $\mathbb{P}^+_{\xi,~y}$, also by the Doob's {\it h}-transform. More precisely, by point 1 of Corollary \ref{Coro1}, for any $N\in \mathbb{N}^*$, $y\geq0$ and $B\in\sigma\left(\xi\right)\otimes\mathscr{F}^{\xi}_N$, we can define $\mathbb{P}^+_{\xi,~y}$ as
\begin{equation}\label{Prob of RW stay positive}
	\mathbb{P}^+_{\xi,~y}(B):=\frac{\mathbb{E}_{\xi,~y}\left(U(\theta^N \xi, S_N)1_{\{\tau_0>N\}}1_B\right)}{U(\xi,y)}.
\end{equation}

For $x\geq0$ with $y=x\sqrt{N}\sigma$, define the rescaled probability laws on $\Omega$ as
\begin{equation}\label{rescaled prob on Omega}
	\mathbf{P}^{N}_{\xi,~x}:=\mathbb{P}_{\xi,~y}\circ (\phi_N)^{-1},~~~~ \mathbf{P}^{+,N}_{\xi,~x}:=\mathbb{P}^+_{\xi,~y}\circ (\phi_N)^{-1}.
\end{equation}

We have the following quenched invariance principle for random walk conditioned to stay positive on the interval $\left(0,\infty\right)$.
\begin{theorem}\label{Th3}
		Assume that (\ref{assumption}) holds, then for almost all $\xi$, we have
		\begin{equation}\label{QIP for random walk conditioned to stay positive}
			\mathbf{P}_{\xi,~x}^{+,N}\Longrightarrow \mathbf{P}_x^+,~~~~\text{as}~~ N\to \infty.
		\end{equation}
\end{theorem}

\begin{remark} 
	
		
The proof of this theorem is divided into two steps.

(1) For the case $x=0$, the main idea is to make use of the absolute continuity between $\mathbf{P}^{+,N}_{\xi}$ $\left(\text{respectively},~\mathbf{P}^{+}\right)$ and the meander $\mathbf{P}^{(m),N}_{\xi}$ $\left(\text{respectively},~\mathbf{P}^{(m)}\right)$ and then to employ quenched invariance principle for the meander (\ref{QIP for meander}). 
	
	
(2) For the case $x>0$, by the same approach, we apply the absolute continuity between $\mathbf{P}^{+,N}_{\xi,~x}$ $\left(\text{respectively},~\mathbf{P}^{+}_{x}\right)$ and $\mathbf{P}_{\xi,~x}^{N}$ $\left(\text{respectively},~\mathbf{P}_{x}\right)$ together with the quenched invariance principle of unconditional law. In fact, under our assumption, for almost all $\xi$, the quenched Lindeberg condition holds, so we can obtain the quenched invariance principle for random walk without conditioning from Theorem 3.1 of Prokhorov \cite{Pro56}. That is, for almost all $\xi$, we have
	\begin{equation}\label{QIP for random walk without conditioning}
		\mathbf{P}_{\xi,~x}^{N}\Longrightarrow \mathbf{P}_x,~~~~\text{as}~~ N\to \infty.
	\end{equation}


\end{remark}


\subsection{Organization of the paper}
The rest of this paper is organized as follows. In Section 2, we prove Theorem \ref{Th1} that the function $U$ given by Proposition \ref{Prop1} is quenched harmonic, and collect some preliminary facts that will be used later. With the help of quenched harmonic function, we can construct the random walks conditioned to stay positive on the interval $\left(0,\infty\right)$ for almost all $\xi$. Then, we prove the QIP for this conditioned walks (Theorem \ref{Th3}) in Section 3.


\section{Proof of the facts for quenched harmonic function}

The main purpose of this section is to prove Theorem \ref{Th1}, we assume (\ref{assumption}) holds throughout this section.

\subsection{Proof of Proposition \ref{Prop1}}

In this subsection, we shall show that the function $U$ is well defined for almost all $\xi$. A key step is to prove that the random variable $S_{\tau_{y}}$ is integrable under $\mathbb{P}_{\xi}$ for almost all $\xi$ (see Lemma \ref{Lem2.2}).

We first give a fact that $\mathbb{P}_{\xi}\left(\tau_{y}<\infty\right)=1$ for almost all $\xi$ (Lemma \ref{Lem2.1}), by the method of strong approximation based on the following implicit rate of convergence in the invariance principle.

\begin{lemma}\label{Lem2.1.1}
	Given the environment $\xi$, for every $n\geq1$, we can  reconstruct a random walk $\{S_n\}_{n\in\mathbb{N}}$ and a Brownian motion $\{W^n_t\}_{t\in\mathbb{R}_+}$ on a common probability space such that
	\begin{equation}\label{strong approximation}
		\mathbf{P}_{\xi}\left(\sup_{0\leq t\leq1}\left|S_{(nt)}-W^n_{n\sigma^2 t}\right|>\sqrt{n}\sigma\pi_{n}\right) \leq \pi_{n},
	\end{equation}
where $\pi_{n}$ (depending on $\xi$) denotes the Prokhorov distance between the distribution on $C[0,1]$ of the process $S_{(\cdot)}$ (the linear interpolation of $S$) and the Brownian motion $W^n_{\cdot}$, and for almost each realization of $\xi$, $\pi_{n}\to0$ as $n\to\infty$.
\end{lemma}
This result is a consequence of Lemma 16 in \cite{DSW18} and Remark 2 of Sakhanenko \cite{Sak06}. Noting that $B_{n}\sim \sqrt{n}\sigma$ $P$-a.s. as $n\to\infty$, and assumption (\ref{assumption}) implies that, for almost each realization of $\xi$, the quenched Lindeberg condition holds, which is equivalent to the Prokhorov distance $\pi_{n}\to0$ as $n\to\infty$.

\begin{lemma}\label{Lem2.1}
	For almost all $\xi$, we have $\tau_{y}<\infty$, $\mathbb{P}_\xi$-a.s..
\end{lemma}

\begin{proof}
	Since $\left|\inf_{t\leq 1}S_{(nt)}-\inf_{t\leq 1}W^n_{n\sigma^2 t}\right|\leq\sup_{t\leq 1}\left|S_{(nt)}-W^n_{n\sigma^2 t}\right|$, we have
	\begin{equation}\nonumber
		\begin{aligned}
			\mathbb{P}_{\xi}\left(\tau_{y}>n\right)=&\mathbb{P}_{\xi}\left(y+\min_{k\leq n}S_{k}>0\right)\\
			=&\mathbf{P}_{\xi}\left(y+\inf_{t\leq 1}S_{(nt)}>0\right)\\
			\leq& \mathbf{P}_{\xi}\left(y+\sup_{t\leq 1}\left|S_{(nt)}-W^n_{n\sigma^2 t}\right|+\inf_{t\leq 1}W^n_{n\sigma^2 t}>0\right)\\
			\leq& \mathbf{P}_{\xi}\left(y+\sqrt{n}\sigma\pi_{n}+\inf_{t\leq 1}W^n_{n\sigma^2 t}>0\right)+\mathbf{P}_{\xi}\left(\sup_{t\leq 1}\left|S_{(nt)}-W^n_{n\sigma^2 t}\right|>\sqrt{n}\sigma\pi_{n}\right)\\
			\leq& \mathbf{P}_{\xi}\left(\frac{y+\sqrt{n}\sigma\pi_{n}}{\sqrt{n}\sigma}+\inf_{t\leq 1}W^n_{t}>0\right)+\pi_{n}\\
			=&2\int_{0}^{\frac{y+\sqrt{n}\sigma\pi_{n}}{\sqrt{n}\sigma}}\frac{1}{\sqrt{2\pi}}e^{-\frac{x^2}{2}}dx+\pi_{n},
		\end{aligned}
	\end{equation}
	where the last inequality follows by Lemma \ref{Lem2.1.1}.
	Since for almost each realization of $\xi$, $\pi_{n}\to0$ as $n\to\infty$, we obtain $\mathbb{P}_{\xi}\left(\tau_{y}>n\right)\to0$.
\end{proof}

Next we prove the integrability of $S_{\tau_{y}}$ under $\mathbb{P}_{\xi}$ for almost all $\xi$. We begin by giving an upper bound for $\mathbb{P}_{\xi}\left(\tau_{y}>n\right)$.

\begin{lemma}\label{Lem2.2.1}
	Fix $y\geq0$, for almost all $\xi$, there exists an integer $N_1$ (depending on $\xi$ and $y$) such that, for all\footnote{Here $n$ depends on $\xi$, but we only write $n$ for simplicity.} $n>N_1$, we have
	$$\mathbb{P}_{\xi}\left(\tau_{y}>n\right)<\frac{3\mathbb{E}_\xi(y+S_n; \tau_{y}>n)}{\sqrt{n}\sigma}.$$
\end{lemma}

\begin{proof}
	For any $n\geq1$, $x\in\mathbb{R}$, we have the inequality \begin{equation}\label{FKG inequality}
		\mathbb{P}_{\xi}(S_n>x,\tau_{y}>n)\geq\mathbb{P}_{\xi}(S_n>x)\mathbb{P}_{\xi}(\tau_{y}>n).
	\end{equation}
	This statement is a direct consequence of FKG inequality, or see Lemma 24 in \cite{DSW18}.
	
	Hence, by (\ref{FKG inequality}), we get
	$$
	\begin{aligned}
		\frac{\mathbb{E}_{\xi}(y+S_n;\tau_{y}>n)}{\mathbb{P}_{\xi}(\tau_{y}>n)}&=\mathbb{E}_{\xi}(y+S_n\mid\tau_{y}>n)\\
		&=\int_{0}^{\infty}\mathbb{P}_{\xi}(y+S_n>x\mid\tau_{y}>n)dx\\
		&\geq\int_{0}^{\infty}\mathbb{P}_{\xi}(y+S_n>x)dx\\
		&=\mathbb{E}_{\xi}(y+S_n)^+.
	\end{aligned}
	$$
	On the other hand, by the central limit theorem, for almost all $\xi$, $\frac{y+S_n}{\sqrt{n}\sigma}$ converges in law ($\mathbb{P}_{\xi}$) to the standard normal distribution. By the Fatou's lemma, we obtain $$\liminf_{n\to\infty}\frac{\mathbb{E}_{\xi}(y+S_n)^+}{\sqrt{n}\sigma}\geq\int_{0}^{\infty}\frac{x}{\sqrt{2\pi}}e^{-\frac{x^2}{2}}dx=\frac{1}{\sqrt{2\pi}}>\frac{1}{3},$$
	so we can choose $N_1:=\max\{n\geq1:\mathbb{E}_{\xi}(y+S_n)^+\leq\frac{\sqrt{n}\sigma}{3}\}<\infty$ for almost all $\xi$.
	
	Therefore, for all $n>N_1$, we have $$\mathbb{P}_{\xi}(\tau_{y}>n)\leq\frac{\mathbb{E}_{\xi}(y+S_n;\tau_{y}>n)}{\mathbb{E}_{\xi}(y+S_n)^+}\leq\frac{3\mathbb{E}_{\xi}(y+S_n;\tau_{y}>n)}{\sqrt{n}\sigma}.$$
\end{proof}

For simplicity, let us introduce some notations:
$$
\begin{aligned}
	&g_n:=\frac{\sqrt{n}}{\log ^{1+\delta}n}~~(\delta>0),~~~~ G_{(>n)}:=\sum_{i>n}\frac{g_i}{(i-1)^{\frac{3}{2}}\sigma},~~~~Y_{(>n)}:=\sum_{i>n}\frac{\mathbb{E}_{\xi}(-X_i;-X_i>g_i)}{\sqrt{i-1}\sigma},\\
	&Z_{k,n}:=\max_{k\leq i\leq n}\mathbb{E}_{\xi}(y+S_i;\tau_{y}>i).
\end{aligned}
$$

\begin{lemma}\label{Lem2.2.2}
	Fix $y\geq0$, let $N_1$ be the same as Lemma \ref{Lem2.2.1}, then for almost all $\xi$, for\footnote{As Lemma \ref{Lem2.2.1}, $N_2, k, n$ all depend on $\xi$.} $N_1\leq N_2\leq k\leq n$, we have
	\begin{equation}\label{N21}
		\mathbb{E}_{\xi}(g_{\tau_{y}},N_2<\tau_{y}\leq k)\leq 3Z_{N_2,n}G_{(>N_2)}
	\end{equation}
and
	\begin{equation}\label{N22}
		\mathbb{E}_{\xi}(-X_{\tau_{y}};-X_{\tau_{y}}>g_{\tau_{y}},N_2<\tau_{y}\leq k)\leq 3Z_{N_2,n}Y_{(>N_2)}.
	\end{equation}
\end{lemma}

\begin{proof}
	To prove $\left(\ref{N21}\right)$. Note that
	\begin{equation}\nonumber
		\begin{aligned}
			\mathbb{E}_{\xi}(g_{\tau_{y}},N_2<\tau_{y}\leq k)=&\sum_{i=N_{2}+1}^{k}g_{i}\mathbb{P}_{\xi}(\tau_{y}=i)\\
			=&\sum_{i=N_{2}+1}^{k}g_{i}\left[\mathbb{P}_{\xi}(\tau_{y}>i-1)-\mathbb{P}_{\xi}(\tau_{y}>i)\right]\\
			=&g_{N_2+1}\mathbb{P}_{\xi}(\tau_{y}>N_2)-g_k\mathbb{P}_{\xi}(\tau_{y}>k)+\sum_{i=N_{2}+1}^{k-1}\left(g_{i+1}-g_i\right)\mathbb{P}_{\xi}(\tau_{y}>i)\\
			\leq&g_{N_2+1}\mathbb{P}_{\xi}(\tau_{y}>N_2)+\sum_{i=N_{2}+1}^{k-1}\left(g_{i+1}-g_i\right)\mathbb{P}_{\xi}(\tau_{y}>i).
		\end{aligned}
	\end{equation}
	Then, by Lemma \ref{Lem2.2.1}, the definition of $Z_{k,n}$ and $G_{(>n)}$, we obtain
	\begin{equation}\nonumber
		\begin{aligned}
			&\mathbb{E}_{\xi}(g_{\tau_{y}},N_2<\tau_{y}\leq k)\\
			\leq&g_{N_2+1}\frac{3\mathbb{E}_{\xi}(y+S_{N_2};\tau_{y}>N_2)}{\sqrt{N_2}\sigma}+\sum_{i=N_{2}+1}^{k-1}\left(g_{i+1}-g_i\right)\frac{3\mathbb{E}_{\xi}(y+S_i;\tau_{y}>i)}{\sqrt{i}\sigma}\\
			\leq&g_{N_2+1}\frac{3Z_{N_2,n}}{\sqrt{N_2}\sigma}+\sum_{i=N_{2}+1}^{k-1}\left(g_{i+1}-g_i\right)\frac{3Z_{N_2,n}}{\sqrt{i}\sigma}\\
			=&\sum_{i=N_{2}+1}^{k}g_i\left(\frac{3Z_{N_2,n}}{\sqrt{i-1}\sigma}-\frac{3Z_{N_2,n}}{\sqrt{i}\sigma}\right)+g_{k}\frac{3Z_{N_2,n}}{\sqrt{k}\sigma}\\
			\leq&\sum_{i>N_{2}}g_i\left(\frac{3Z_{N_2,n}}{\sqrt{i-1}\sigma}-\frac{3Z_{N_2,n}}{\sqrt{i}\sigma}\right)\\
			=&3Z_{N_2,n}\sum_{i>N_{2}}\frac{g_i}{\sqrt{i-1}\sqrt{i}(\sqrt{i-1}+\sqrt{i})\sigma}\\
			\leq&3Z_{N_2,n}\sum_{i>N_{2}}\frac{g_i}{(i-1)^{\frac{3}{2}}\sigma}\\
			=&3Z_{N_2,n}G_{(>N_2)},
		\end{aligned}
	\end{equation}
	where the second to last inequality holds since $g_n$ is positive and increasing in $n$.
	
	To prove $\left(\ref{N21}\right)$. Observe that
	\begin{equation}\nonumber
		\begin{aligned}
			\mathbb{E}_{\xi}(-X_{\tau_{y}};-X_{\tau_{y}}>g_{\tau_{y}},N_2<\tau_{y}\leq k)=&\sum_{i=N_{2}+1}^{k}\mathbb{E}_{\xi}(-X_{i};-X_{i}>g_{i},\tau_{y}=i)\\
			\leq&\sum_{i=N_{2}+1}^{k}\mathbb{E}_{\xi}(-X_{i};-X_{i}>g_{i},\tau_{y}>i-1)\\
			=&\sum_{i=N_{2}+1}^{k}\mathbb{E}_{\xi}(-X_{i};-X_{i}>g_{i})\mathbb{P}_{\xi}(\tau_{y}>i-1).
		\end{aligned}
	\end{equation}
	Using again Lemma \ref{Lem2.2.1}, the definition of $Z_{k,n}$ and $Y_{(>n)}$, we get
	\begin{equation}\nonumber
		\begin{aligned}
			\mathbb{E}_{\xi}(-X_{\tau_{y}};-X_{\tau_{y}}>g_{\tau_{y}},N_2<\tau_{y}\leq k)\leq&\sum_{i=N_{2}+1}^{k}\mathbb{E}_{\xi}(-X_{i};-X_{i}>g_{i})\frac{3\mathbb{E}_{\xi}(y+S_{i-1};\tau_{y}>i-1)}{\sqrt{i-1}\sigma}\\
			\leq&3Z_{N_2,n}\sum_{i>N_{2}}\frac{\mathbb{E}_{\xi}(-X_{i};-X_{i}>g_{i})}{\sqrt{i-1}\sigma}\\
			=&3Z_{N_2,n}Y_{(>N_2)}.
		\end{aligned}
	\end{equation}
\end{proof}

\begin{lemma}\label{Lem2.2}
	For any $y\geq0$, we have $\mathbb{E}_\xi |S_{\tau_{y}}|<\infty$, $P$-a.s..
\end{lemma}

\begin{proof}
	First note that since $\sum_{i=2}^{\infty}\frac{g_i}{(i-1)^{\frac{3}{2}}\sigma}<\infty$, we get $G_{(>n)}\to0$ as $n\to\infty$, and (\ref{assumption}) implies that, for almost all $\xi$, $Y_{(>n)}\to0$ as $n\to\infty$. In fact, by (\ref{assumption}), we get $\sum_{i=1}^{\infty}\mathbb{P}\left(|X_i|>i^{\frac{1}{2+\epsilon}}\right)<\infty$. Hence, $\sum_{i=1}^{\infty}\mathbb{P}\left(-X_i>g_i\right)<\infty$. So, for almost all $\xi$, $\sum_{i=1}^{\infty}\mathbb{P}_{\xi}\left(-X_i>g_i\right)<\infty$. Since by H\"{o}lder inequality (let $p=1+\frac{\epsilon}{2}$ and $\frac{1}{p}+\frac{1}{q}=1$), we obtain $\mathbb{E}_{\xi}(-X_i;-X_i>g_i)\leq\mathbb{E}_{\xi}(\frac{|X_i|^2}{g_i};-X_i>g_i)\leq\frac{\left(\mathbb{E}_{\xi}(|X_i|^{2p})\right)^{\frac{1}{p}}}{g_i}\left(\mathbb{P}_{\xi}(-X_i>g_i)\right)^{\frac{1}{q}}$. Thus, for almost all $\xi$, we have  $\sum_{i>n}\frac{\mathbb{E}_{\xi}(-X_i;-X_i>g_i)}{\sqrt{i-1}\sigma}<\infty$, this implies $Y_{(>n)}\to0$ as $n\to\infty$. As a result, for almost all $\xi$, we can choose $N_2:=\min\{n>N_1:G_{(>n)}+Y_{(>n)}\leq\frac{1}{6}\}<\infty$.
	
	Next, since $y+S_{i-1}>0$ in the event $\{\tau_{y}=i\}$, we get $$-y-S_{\tau_{y}}=-y-S_{\tau_{y}-1}-X_{\tau_{y}}<-X_{\tau_{y}}\leq g_{\tau_{y}}-X_{\tau_{y}}1_{\{-X_{\tau_{y}}>g_{\tau_{y}}\}}.$$
	From this inequality and Lemma \ref{Lem2.2.2}, for $N_2\leq k\leq n$, we obtain
	$$
	\begin{aligned}
		&\mathbb{E}_{\xi}\left(-y-S_{\tau_{y}};\tau_{y}\leq k\right)\\
		\leq &\mathbb{E}_{\xi}\left(-y-S_{\tau_{y}};\tau_{y}\leq N_2\right)+\mathbb{E}_{\xi}\left(g_{\tau_{y}};N_2<\tau_{y}\leq k\right)+\mathbb{E}_{\xi}(-X_{\tau_{y}};-X_{\tau_{y}}>g_{\tau_{y}},N_2<\tau_{y}\leq k)\\
		\leq &C_1+3Z_{N_2,n}G_{(>N_2)}+3Z_{N_2,n}Y_{(>N_2)}\\
		\leq &C_1+\frac{Z_{N_2,n}}{2},
	\end{aligned}
	$$
	where $C_1:=\mathbb{E}_{\xi}\left(-y-S_{\tau_{y}};\tau_{y}\leq N_2\right)$. Note that $C$ depends on $\xi$ and $C<\infty$ for almost all $\xi$. In fact,
	$$
	\mathbb{E}_{\xi}\left(-y-S_{\tau_{y}};\tau_{y}\leq N_2\right)=\sum_{i=1}^{N_2}\mathbb{E}_{\xi}\left(-y-S_{i};\tau_{y}=i\right)\leq\sum_{i=1}^{N_2}\mathbb{E}_{\xi}\left(\left(y+S_{i}\right)^{-}\right)<\infty.
	$$
	
	Then, by optional stopping theorem (see (\ref{optional stopping}) below), for $N_2\leq k\leq n$, we get $$\mathbb{E}_\xi(y+S_{k};\tau_{y}>k)=y+\mathbb{E}_\xi(-y-S_{\tau_{y}};\tau_{y}\leq k)\leq y+C_1+\frac{Z_{N_2,n}}{2}.$$
	Taking maximum in last inequality with respect to $k\in[N_2,n]$, we get $Z_{N_2,n}\leq y+C_1+\frac{Z_{N_2,n}}{2}$, which implies $Z_{N_2,n}\leq 2y+2C_1$.
	
	Therefore, $\mathbb{E}_{\xi}\left(-y-S_{\tau_{y}};\tau_{y}\leq n\right)\leq y+2C_1$. By the monotone convergence theorem, we have $$0\leq\mathbb{E}_{\xi}\left(-y-S_{\tau_{y}}\right)=\lim_{n\to\infty}\mathbb{E}_{\xi}\left(-y-S_{\tau_{y}};\tau_{y}\leq n\right)\leq y+2C_1,$$
	which implies $0\leq-\mathbb{E}_{\xi}\left(S_{\tau_{y}}\right)\leq 2y+2C_1$.
\end{proof}

\begin{proof}[Proof of Proposition \ref{Prop1}]
	Since $\{S_n\}_{n\in\mathbb{N}}$ is a martingale under $\mathbb{P}_\xi$, for any $n\geq1$, by the optional stopping theorem,
	$$0=\mathbb{E}_\xi(S_{n\wedge\tau_{y}})=\mathbb{E}_\xi(S_{\tau_{y}};\tau_{y}\leq n)+\mathbb{E}_\xi(S_{n};\tau_{y}>n).$$ Therefore, 
	From this equality, we obtain
	\begin{equation}\label{optional stopping}
		\mathbb{E}_\xi(y+S_{n};\tau_{y}>n)=y\mathbb{P}_\xi(\tau_{y}>n)-\mathbb{E}_\xi(S_{\tau_{y}};\tau_{y}\leq n).
	\end{equation}
	By Lemma \ref{Lem2.1}, Lemma \ref{Lem2.2} and the dominated convergence theorem, $$\lim_{n\to\infty}\mathbb{E}_\xi(y+S_{n};\tau_{y}>n)=-\mathbb{E}_\xi(S_{\tau_{y}}),$$ Proposition \ref{Prop1} follows.
\end{proof}


\subsection{Proof of Theorem \ref{Th1}}

Although, the proof of the quenched harmonic property is simple: express $U_{n+1}(\xi,y)$ as the integral of $U_{n}(\theta\xi,y')$ by the Markov property, then the quenched harmonic property follows by the dominated convergence theorem. This idea is based on the fact that the random environment is an infinite sequence of random elements. By shifting the environment, we obtain the quenched harmonic property rather than annealed harmonic property.

\begin{proof}[Proof of claim 1 of Theorem \ref{Th1}]
	By the Markov property, and under $\mathbb{P}_{\theta\xi}$, the random walk $S$ is  $S_{n}=\sum_{i=2}^{n+1}X_{i}$, we obtain, for $n\geq1$,
		\begin{equation}\label{harmonic}
			\begin{aligned}
				U_{n+1}(\xi,y)=&\mathbb{E}_\xi\left(y+S_{n+1}; \tau_{y}>n+1\right)\\
				=&\mathbb{E}_\xi\left[\mathbb{E}_\xi\left(y+S_{n+1}; \tau_{y}>n+1\mid \mathscr{F}^{\xi}_{1}\right)\right]\\
				=&\int_{\mathbb{R}}\mathbb{E}_\xi\left(y'+\sum_{i=2}^{n+1}X_{i}; \tau_{y'}>n\mid \mathscr{F}^{\xi}_{1}\right)\mathbb{P}_\xi(y+S_1\in dy'; \tau_{y}>1)\\
				=&\int_{\mathbb{R}}\mathbb{E}_{\theta\xi}\left(y'+S_{n}; \tau_{y'}>n\right)\mathbb{P}_\xi(y+S_1\in dy'; \tau_{y}>1)\\
				=&\int_{\mathbb{R}}U_n(\theta \xi,y')\mathbb{P}_\xi(y+S_1\in dy'; \tau_{y}>1).
			\end{aligned}
	\end{equation}
	
	Note that $U_{n}(\xi,y)$ is non-decreasing in $n$, since the sequence $\{(y+S_n)1_{\{\tau_{y}>n\}}\}$ is a submartingale. In fact, we have
	$$
	\begin{aligned}
		&\mathbb{E}_{\xi}\left(\left(y+S_{n+1}\right)1_{\{\tau_{y}>n+1\}}\mid\mathscr{F}^{\xi}_{n}\right)\\
		=&\mathbb{E}_{\xi}\left(\left(y+S_{n+1}\right)\left(1_{\{\tau_{y}>n\}}-1_{\{\tau_{y}=n+1\}}\right)\mid\mathscr{F}^{\xi}_{n}\right)\\
		=&\mathbb{E}_{\xi}\left(y+S_{n+1}\mid\mathscr{F}^{\xi}_{n}\right)1_{\{\tau_{y}>n\}}-\mathbb{E}_{\xi}\left(\left(y+S_{n+1}\right)1_{\{\tau_{y}=n+1\}}\mid\mathscr{F}^{\xi}_{n}\right)\\
		=&(y+S_{n})1_{\{\tau_{y}>n\}}-\mathbb{E}_{\xi}\left((y+S_{n+1})1_{\{\tau_{y}=n+1\}}\mid\mathscr{F}^{\xi}_{n}\right)\\
		\geq&(y+S_{n})1_{\{\tau_{y}>n\}}
	\end{aligned}
	$$
	where the last inequality follows from the fact that $y+S_{n+1}\leq0$ on the event $\{\tau_{y}=n+1\}$. This yields the submartingale property.
	
	Taking the limit in (\ref{harmonic}), by Lemma \ref{Lem2.2} and the dominated convergence theorem, we get
	\begin{equation}\nonumber
		\begin{aligned}
			U(\xi,y)=&\int_{\mathbb{R}}U(\theta \xi,y')\mathbb{P}_\xi(y+S_1\in dy'; \tau_{y}>1)\\
			=&\mathbb{E}_{\xi}\left[U(\theta \xi,y+S_{1});\tau_{y}>1\right],
		\end{aligned}
	\end{equation}
	claim 1 follows.
\end{proof}

To obtain the asymptotic behaviour of the quenched harmonic function, we construct a non-negative supermartingale in the following lemma.

\begin{lemma}
	For $C\in\mathbb{R}_{+}$ and $K\in\mathbb{N}^*$, let
	\begin{equation}\nonumber
		f(x):=\begin{cases}
			(C+x)^{1-\frac{1}{2K+1}}, &x>0,\\
			|x|, &x\leq0.
		\end{cases}
	\end{equation}
	Then for almost all $\xi$, there exist sufficiently large $C$ and $K$ (both depending on $\xi$) such that $f\left(y+S_{n\wedge\tau_y}\right)$ is a non-negative supermartingale under $\mathbb{P}_\xi$.
\end{lemma}

\begin{proof}
	Fix each realization of $\xi$, we first show the integrability of $f\left(y+S_{n\wedge\tau_y}\right)$ under $\mathbb{P}_\xi$. For each $n$, we have
	$$
	\begin{aligned}
		\mathbb{E}_{\xi}f\left(y+S_{n\wedge\tau_y}\right)&=\mathbb{E}_{\xi}\left(f\left(y+S_{n}\right);\tau_{y}>n\right)+\mathbb{E}_{\xi}\left(f\left(y+S_{\tau_y}\right);\tau_{y}\leq n\right)\\
		&=\mathbb{E}_{\xi}\left((C+y+S_{n})^{1-\frac{1}{2K+1}};\tau_{y}>n\right)+\mathbb{E}_{\xi}\left(\left|y+S_{\tau_y}\right|;\tau_{y}\leq n\right)\\		&\leq\mathbb{E}_{\xi}\left(C+y+S_{n};\tau_{y}>n\right)-\mathbb{E}_{\xi}\left(y+S_{\tau_y};\tau_{y}\leq n\right)\\
		&=C\mathbb{P}_{\xi}\left(\tau_{y}>n\right)+y\mathbb{P}_{\xi}\left(\tau_{y}>n\right)-y\mathbb{P}_{\xi}\left(\tau_{y}\leq n\right)-2\mathbb{E}_{\xi}\left(S_{\tau_y};\tau_{y}\leq n\right)\\
		&\leq C+y-2\mathbb{E}_{\xi}\left(S_{\tau_{y}}\right),
	\end{aligned}
	$$
	where the last equality follows from $\mathbb{E}_\xi(S_{n};\tau_{y}>n)=-\mathbb{E}_\xi(S_{\tau_{y}};\tau_{y}\leq n)$.
	Therefore, by Lemma \ref{Lem2.2}, we obtain that, for almost all $\xi$, $\mathbb{E}_{\xi}f\left(y+S_{n\wedge\tau_y}\right)<\infty$.
	
	Next, we shall prove the supermartingale property, that is,
	\begin{equation}\label{supermartingale property}
		\begin{aligned}
			\mathbb{E}_{\xi}\left[f\left(y+S_{(n+1)\wedge\tau_y}\right)|\mathscr{F}^{\xi}_{n}\right]&\leq f\left(y+S_{n\wedge\tau_y}\right)\\
			&=\left(C+y+S_{n}\right)^{1-\frac{1}{2K+1}}1_{\{\tau_{y}>n\}}+\left|y+S_{\tau_y}\right|1_{\{\tau_{y}\leq n\}}.
		\end{aligned}
	\end{equation}
	By the definition of $\tau_{y}$ and $f(x)$, we get
	\begin{equation}\label{decomposition at tau}
		\begin{aligned}
			\mathbb{E}_{\xi}\left[f\left(y+S_{(n+1)\wedge\tau_y}\right)|\mathscr{F}^{\xi}_{n}\right]=&\mathbb{E}_{\xi}\left[(C+y+S_{n+1})^{1-\frac{1}{2K+1}}1_{\{\tau_{y}>n+1\}}|\mathscr{F}^{\xi}_{n}\right]\\
			&+\mathbb{E}_{\xi}\left[\left|y+S_{\tau_y}\right|1_{\{\tau_{y}\leq n+1\}}|\mathscr{F}^{\xi}_{n}\right].
		\end{aligned}
	\end{equation}
	For the first term in $\left(\ref{decomposition at tau}\right)$, we obtain
	$$
	\begin{aligned}
		&\mathbb{E}_{\xi}\left[(C+y+S_{n+1})^{1-\frac{1}{2K+1}}1_{\{\tau_{y}>n+1\}}|\mathscr{F}^{\xi}_{n}\right]\\
		\leq&\left(\mathbb{E}_{\xi}\left[(C+y+S_{n+1})1_{\{\tau_{y}>n+1\}}|\mathscr{F}^{\xi}_{n}\right]\right)^{1-\frac{1}{2K+1}}\\
		=&\left(\mathbb{E}_{\xi}\left[(C+y+S_{n+1})\left(1_{\{\tau_{y}>n\}}-1_{\{\tau_{y}=n+1\}}\right)|\mathscr{F}^{\xi}_{n}\right]\right)^{1-\frac{1}{2K+1}}\\
		=&\left(\mathbb{E}_{\xi}\left[(C+y+S_{n+1})|\mathscr{F}^{\xi}_{n}\right]\right)^{1-\frac{1}{2K+1}}1_{\{\tau_{y}>n\}}-\left(\mathbb{E}_{\xi}\left[(C+y+S_{n+1})1_{\{\tau_{y}=n+1\}}|\mathscr{F}^{\xi}_{n}\right]\right)^{1-\frac{1}{2K+1}}\\
		=&\left(C+y+S_{n}\right)^{1-\frac{1}{2K+1}}1_{\{\tau_{y}>n\}}-\left(\mathbb{E}_{\xi}\left[(C+y+S_{n+1})1_{\{\tau_{y}=n+1\}}|\mathscr{F}^{\xi}_{n}\right]\right)^{1-\frac{1}{2K+1}},
	\end{aligned}
	$$
	where the first inequality follows by the Jensen's inequality.	For the second term in $\left(\ref{decomposition at tau}\right)$, we have
	$$
	\begin{aligned}
		&\mathbb{E}_{\xi}\left[\left|y+S_{\tau_y}\right|1_{\{\tau_{y}\leq n+1\}}|\mathscr{F}^{\xi}_{n}\right]\\
		=&\mathbb{E}_{\xi}\left[\left|y+S_{\tau_y}\right|\left(1_{\{\tau_{y}\leq n\}}+1_{\{\tau_{y}=n+1\}}\right)|\mathscr{F}^{\xi}_{n}\right]\\
		=&\left|y+S_{\tau_y}\right|1_{\{\tau_{y}\leq n\}}+\mathbb{E}_{\xi}\left[\left|y+S_{n+1}\right|1_{\{\tau_{y}=n+1\}}|\mathscr{F}^{\xi}_{n}\right].
	\end{aligned}
	$$
	
	Now we can choose $C$ and $K$ sufficiently large such that $$\mathbb{E}_{\xi}\left[\left|y+S_{n+1}\right|1_{\{\tau_{y}=n+1\}}|\mathscr{F}^{\xi}_{n}\right]
	-\left(\mathbb{E}_{\xi}\left[(C+y+S_{n+1})1_{\{\tau_{y}=n+1\}}|\mathscr{F}^{\xi}_{n}\right]\right)^{1-\frac{1}{2K+1}}\leq0.$$
	In fact, for any $0<\eta<1$, we can choose $N_3$ such that $$\mathbb{E}_{\xi}\left(|X_{n+1}|1_{\{|X_{n+1}|>N_3\}}\right)\leq\sum_{k=N_3}^{\infty}\mathbb{P}_{\xi}(|X_{n+1}|\geq k)<\eta.$$
	Then, we let $C>2N_3$ and $K$ large enough, it follows that $$\mathbb{E}_{\xi}\left[\left|y+S_{n}+X_{n+1}\right|1_{\{\tau_{y}=n+1\}}|\mathscr{F}^{\xi}_{n}\right]
	\leq \left(\mathbb{E}_{\xi}\left[(C+y+S_{n}+X_{n+1})1_{\{\tau_{y}=n+1\}}|\mathscr{F}^{\xi}_{n}\right]\right)^{1-\frac{1}{2K+1}}.$$
	
	Therefore, we obtain (\ref{supermartingale property}) immediately from the above arguments.
\end{proof}

\begin{proof}[Proof of claim 2 of Theorem \ref{Th1}]
	Note that $U(\xi,0)>0$ for almost all $\xi$ under the assumption $\mathbb{P}_\xi(X_1>0)>0, ~P\text{-a.s.}$. For any $y>0$, since $y+S_{\tau_{y}}\leq0$, by (\ref{optional stopping}) and the definition of $\tau_{y}$, we have, for any $n\geq1$, $$\mathbb{E}_\xi(y+S_{n};\tau_{y}>n)=y-\mathbb{E}_\xi(y+S_{\tau_{y}};\tau_{y}\leq n)\geq y.$$ Taking the limit as $n\to\infty$ and using Proposition \ref{Prop1}, we get $U(\xi,y)\geq y>0$.
	
	For any $y_1\leq y_2$, then $\tau_{y_1}\leq \tau_{y_2}$. Therefore, for any $n\geq1$, $$\mathbb{E}_\xi(y_1+S_{n};\tau_{y_1}>n)\leq \mathbb{E}_\xi(y_2+S_{n};\tau_{y_1}>n)\leq \mathbb{E}_\xi(y_2+S_{n};\tau_{y_2}>n)$$ Taking the limit again, we get $U(\xi,y_1)\leq U(\xi,y_2)$.
	
	Since $f(y+S_{n\wedge\tau_y})$ is a non-negative supermartingale, for $y>0$, by the optional stopping theorem,
	\begin{equation}\nonumber
		\mathbb{E}_\xi|y+S_{\tau_y}|\leq (C+y)^{1-\frac{1}{2K+1}}.
	\end{equation}
	Indeed, by the Fatou's lemma and supermartingale property,
		$$
		\begin{aligned}
			\mathbb{E}_{\xi}|y+S_{\tau_y}|=&\mathbb{E}_{\xi}\left[\lim_{n\to\infty}f\left(y+S_{n\wedge\tau_y}\right)\right]\\
			\leq&\liminf_{n\to\infty}\mathbb{E}_{\xi}\left[f\left(y+S_{n\wedge\tau_y}\right)\right]\\
			\leq&\mathbb{E}_{\xi}\left[f\left(y+S_{0}\right)\right]\\
			=&(C+y)^{1-\frac{1}{2K+1}}.
		\end{aligned}
		$$
	This implies
	\begin{equation}\nonumber
		\lim\limits_{y\to\infty}\frac{-\mathbb{E}_\xi(S_{\tau_y})}{y}=1.
	\end{equation}
	The proof of Theorem \ref{Th1} is completed.
\end{proof}


\subsection{Proof of Corollary \ref{Coro1}}

\begin{proof}[Proof of Corollary \ref{Coro1}]
	\textit{Proof of claim 1.} For any $B\in\sigma\left(\xi\right)\otimes\mathscr{F}^{\xi}_{n}$, by the quenched harmonic property and independence, we get
	\begin{equation}\nonumber
		\mathbb{E}_\xi\left[U(\theta^{n+1}\xi,y+S_{n+1})1_{\{\tau_{y}>n+1\}}1_B\right]=\mathbb{E}_\xi\left[U(\theta^n\xi,y+S_n)1_{\{\tau_{y}>n\}}1_B\right],
	\end{equation}
	which implies the martingale property.
	
	\textit{Proof of claim 2.} It follows from (\ref{asymptotic behaviour of harmonic function 1}) that for each $n\in\mathbb{N^*}$ and $y\geq0$,
	\begin{equation}\nonumber
		\lim\limits_{y\to\infty}\frac{U(\theta^n\xi,y)}{y}=1,~~~~P\text{-a.s.}.
	\end{equation}
	Using the Egorov's theorem, we have the uniform convergence with positive probability, hence $$P\left(\lim\limits_{n\to\infty}\frac{U(\theta^n\xi,y_n)}{y_n}=1\right)>0.$$
	Note that $\left\{\lim\limits_{n\to\infty}\frac{U(\theta^n\xi,y_n)}{y_n}=1\right\}$ is $\theta$-invariant set and $\theta$ is ergodic by hypothesis, we obtain
	$$P\left(\lim\limits_{n\to\infty}\frac{U(\theta^n\xi,y_n)}{y_n}=1\right)=1.$$
	
	\textit{Proof of claim 3.} In our situation, for almost every realization of $\xi$, the conditions of Theorem 1 in \cite{DSW18} hold, and $B_{n}\sim \sqrt{n}\sigma$ $P$-a.s. as $n\to\infty$, we can obtain this claim by the same argument developed there.
\end{proof}


\section{Proof of Theorem \ref{Th3}}

\subsection{The case $x=0$: Convergence of $\mathbf{P}^{+,N}_{\xi}$}

In this section, we prove Theorem \ref{Th3} for the case $x=0$, that is, for almost all $\xi$, we have
$\mathbf{P}_{\xi}^{+,N}\Longrightarrow \mathbf{P}^+,~\text{as}~ N\to \infty.$ By the same argument as Theorem 8.1.11 of Durrett \cite{Dur19}, we only need to prove that the sequence of $\mathbf{P}_{\xi}^{+,N}$ converges in law to $\mathbf{P}^+$ on $\Omega_1$ for almost all $\xi$. It is worth noting that we may sometimes omit the word ``for almost all $\xi$'' for the sake of brevity.

We first obtain the absolute continuity between $\mathbf{P}^{+,N}_{\xi}$ restricted to $\Omega_1$ and the meander $\mathbf{P}^{(m),N}_{\xi}$. That is, for any $A\in\mathscr{F}_1$,
\begin{equation}\label{absolute continuity 1}
	\mathbf{P}^{+,N}_{\xi}(A)=\mathbf{E}^{(m),N}_{\xi}\left(\tilde{U}(\xi,N,W_1)1_A\right),
\end{equation} where
\begin{equation}\label{rescaled harmonic function}
	\tilde{U}(\xi,N,x):=\frac{U(\theta^N\xi,\sqrt{N}\sigma x)\mathbb{P}_\xi(\tau_{0}>N)}{U(\xi,0)}.
\end{equation}
In fact, on the one hand, by (\ref{Prob of RW stay positive up to N}), (\ref{Prob of RW stay positive}) and (\ref{rescaled prob on Omega}), we have
$$
\begin{aligned}
	\mathbf{P}_{\xi}^{+,N}(A)=&\mathbb{P}^+_{\xi}\left((\phi_N)^{-1}(A)\right)\\
=&\frac{\mathbb{E}_{\xi}\left(U(\theta^N \xi, S_N)1_{\{\tau_0>N\}}1_{(\phi_N)^{-1}(A)}\right)}{U(\xi,0)}\\
	=&\frac{\mathbb{E}_{\xi}\left(U(\theta^N \xi, S_N)1_{(\phi_N)^{-1}(A)}| \tau_0>N \right) \mathbb{P}_\xi(\tau_0>N)}{U(\xi,0)}\\
	=&\frac{\mathbb{E}_{\xi}^{(m),N}\left(U(\theta^N \xi, S_N)1_{(\phi_N)^{-1}(A)}\right)\mathbb{P}_\xi(\tau_0>N)}{U(\xi,0)}.
\end{aligned}
$$
On the other hand, by (\ref{rescaled meander measure on Omega_1}) and (\ref{rescaled harmonic function}), we have
$$
\begin{aligned}
	\mathbf{E}^{(m),N}_{\xi}\left(\tilde{U}(\xi,N,W_1)1_A\right)=&\mathbb{E}^{(m),N}_{\xi}\left(\tilde{U}(\xi,N,\frac{S_N}{\sqrt{N}\sigma})1_{(\phi_N)^{-1}(A)}\right)\\
	=&\frac{\mathbb{E}_{\xi}^{(m),N}\left(U(\theta^N \xi, S_N)1_{(\phi_N)^{-1}(A)}\right)\mathbb{P}_\xi(\tau_0>N)}{U(\xi,0)}.
\end{aligned}
$$
Hence, (\ref{absolute continuity 1}) follows.

Recall that the absolute continuity relation holds between $\mathbf{P}^{+}$ restricted to $\Omega_1$ and the meander $\mathbf{P}^{(m)}$. That is, for any $A\in\mathscr{F}_1$,
\begin{equation}\label{absolute continuity 2}
	\mathbf{P}^{+}(A)=\mathbf{E}^{(m)}\left(\sqrt{\frac{2}{\pi}}V(W_1)1_A\right),
\end{equation} where $V(x)$ is the harmonic function of killed Brownian motion as in (\ref{Prob of Bessel process}).

We then prove that $\mathbf{P}_{\xi}^{+,N}\Longrightarrow \mathbf{P}^+$ as $N\to \infty$. That is we show that, for any bounded and continuous functional $H$ on $\Omega_1$, $\mathbf{E}_{\xi}^{+,N}(H)\to \mathbf{E}^+(H)$ as $N\to\infty$. Thanks to (\ref{absolute continuity 1}) and (\ref{absolute continuity 2}), this is equivalent to prove that
\begin{equation}\label{weak convergence of meander}
	\mathbf{E}^{(m),N}_{\xi}\left(\tilde{U}(\xi,N,W_1)H\right)\longrightarrow \mathbf{E}^{(m)}\left(\sqrt{\frac{2}{\pi}}V(W_1)H\right), ~~~~as~~N\to\infty.
\end{equation}

We are devoted to proving (\ref{weak convergence of meander}) by a truncation argument and the quenched invariance principle (\ref{QIP for meander}). Let $L>0$, we define the truncation function $I_{L}(x)$ as
\begin{equation}\label{truncation function}
	I_{L}(x):=\begin{cases}
		0, &x\geq L+1;\\
		L+1-x, &L<x<L+1;\\
		1, &x\leq L.
	\end{cases}
\end{equation}

Rewrite both sides of (\ref{weak convergence of meander}) as
$$
\begin{aligned}
	&\mathbf{E}^{(m),N}_{\xi}\left(\tilde{U}(\xi,N,W_1)H\right)\\
	=&\mathbf{E}^{(m),N}_{\xi}\left(\tilde{U}(\xi,N,W_1)HI_{L}\left(W_{1}\right)\right)+\mathbf{E}^{(m),N}_{\xi}\left(\tilde{U}(\xi,N,W_1)H\left(1-I_{L}\left(W_{1}\right)\right)\right)
\end{aligned}
$$
and
$$
\begin{aligned}
	&\mathbf{E}^{(m)}\left(\sqrt{\frac{2}{\pi}}V(W_1)H\right)\\
	=&\mathbf{E}^{(m)}\left(\sqrt{\frac{2}{\pi}}V(W_1)HI_{L}\left(W_{1}\right)\right)+\mathbf{E}^{(m)}\left(\sqrt{\frac{2}{\pi}}V(W_1)H\left(1-I_{L}\left(W_{1}\right)\right)\right).
\end{aligned}
$$
Using the triangle inequality, we have
$$
\begin{aligned}
	&\left| \mathbf{E}^{(m),N}_{\xi} \left(\tilde{U}(\xi,N,W_1)H\right)-\mathbf{E}^{(m)}\left(\sqrt{\frac{2}{\pi}}V(W_1)H\right) \right| \\
	\leq &\left| \mathbf{E}^{(m),N}_{\xi}\left(\tilde{U}(\xi,N,W_1)H I_{L}\left(W_{1}\right)\right)-\mathbf{E}^{(m)}\left(\sqrt{\frac{2}{\pi}}V(W_1)H I_{L}\left(W_{1}\right)\right)\right| \\
	&+\left|	\mathbf{E}^{(m),N}_{\xi}\left(\tilde{U}(\xi,N,W_1)H\left(1-I_{L}\left(W_{1}\right)\right)\right)\right|+\left|\mathbf{E}^{(m)}\left(\sqrt{\frac{2}{\pi}}V(W_1)H\left(1-I_{L}\left(W_{1}\right)\right)\right)\right|\\
	\leq &\left| \mathbf{E}^{(m),N}_{\xi}\left(\tilde{U}(\xi,N,W_1)H I_{L}\left(W_{1}\right)\right)-\mathbf{E}^{(m)}\left(\sqrt{\frac{2}{\pi}}V(W_1)H I_{L}\left(W_{1}\right)\right)\right| \\
	&+M\mathbf{E}^{(m),N}_{\xi}\left(\tilde{U}(\xi,N,W_1)\left(1-I_{L}\left(W_{1}\right)\right)\right)+M\mathbf{E}^{(m)}\left(\sqrt{\frac{2}{\pi}}V(W_1)\left(1-I_{L}\left(W_{1}\right)\right)\right),
\end{aligned}
$$
where the last inequality follows from the boundedness of $H$ (say $|H|\leq M$) and the non-negativity of the terms $\tilde{U}(\xi,N,W_1)\left(1-I_{L}\left(W_{1}\right)\right)$ and $\sqrt{\frac{2}{\pi}}V(W_1)\left(1-I_{L}\left(W_{1}\right)\right)$. By (\ref{absolute continuity 1}) and (\ref{absolute continuity 2}), we have $	\mathbf{E}^{(m),N}_{\xi}\left(\tilde{U}(\xi,N,W_1)\right)=1$ and $\mathbf{E}^{(m)}\left(\sqrt{\frac{2}{\pi}}V(W_1)\right)=1$, then
\begin{equation}\label{triangle inequality}
	\begin{aligned}
		&\left| \mathbf{E}^{(m),N}_{\xi} \left(\tilde{U}(\xi,N,W_1)H\right)-\mathbf{E}^{(m)}\left(\sqrt{\frac{2}{\pi}}V(W_1)H\right) \right| \\
		\leq &\left| \mathbf{E}^{(m),N}_{\xi}\left(\tilde{U}(\xi,N,W_1)H I_{L}\left(W_{1}\right)\right)-\mathbf{E}^{(m)}\left(\sqrt{\frac{2}{\pi}}V(W_1)H I_{L}\left(W_{1}\right)\right)\right| \\
		&+M\left(1-	\mathbf{E}^{(m),N}_{\xi}\left(\tilde{U}(\xi,N,W_1) I_{L}\left(W_{1}\right)\right)\right) \\
		&+M\left(1-\mathbf{E}^{(m)}\left(\sqrt{\frac{2}{\pi}}V(W_1) I_{L}\left(W_{1}\right)\right)\right).
	\end{aligned}
\end{equation}

We claim that for all $L>0$, as $N \rightarrow \infty$,
\begin{equation}\label{first term of triangle inequality}
	\left| \mathbf{E}^{(m),N}_{\xi}\left(\tilde{U}(\xi,N,W_1)H I_{L}\left(W_{1}\right)\right)-\mathbf{E}^{(m)}\left(\sqrt{\frac{2}{\pi}}V(W_1)H I_{L}\left(W_{1}\right)\right)\right| \rightarrow 0.
\end{equation}
We postpone the proof of $\left(\ref{first term of triangle inequality}\right)$ and return to $\left(\ref{triangle inequality}\right)$. It follows that the second term in the right hand side of (\ref{triangle inequality}) converges to the third term by taking $H \equiv 1$ in $\left(\ref{first term of triangle inequality}\right)$, and let $L\to\infty$, the third term vanished since $\mathbf{E}^{(m)}\left(\sqrt{\frac{2}{\pi}}V(W_1)\right)=1$. Therefore, the term in the left hand side of $\left(\ref{triangle inequality}\right)$ tends to zero as $N\to\infty$, (\ref{weak convergence of meander}) follows.

It remains to prove (\ref{first term of triangle inequality}). By the triangle inequality again, we have
\begin{equation}\label{triangle inequality 2}
	\begin{aligned}
		&\left|	\mathbf{E}^{(m),N}_{\xi}\left(\tilde{U}(\xi,N,W_1)H I_{L}\left(W_{1}\right)\right)-\mathbf{E}^{(m)}\left(\sqrt{\frac{2}{\pi}}V(W_1)H I_{L}\left(W_{1}\right)\right)\right| \\
		\leq&M\sup_{x\in(0,L]}\left|\tilde{U}(\xi,N,x)-\sqrt{\frac{2}{\pi}}V(x)\right|\\
		&+\left| \mathbf{E}^{(m),N}_{\xi}\left(\sqrt{\frac{2}{\pi}}V(W_1)H I_{L}\left(W_{1}\right)\right)-\mathbf{E}^{(m)}\left(\sqrt{\frac{2}{\pi}}V(W_1)H I_{L}\left(W_{1}\right)\right)\right|.
	\end{aligned}
\end{equation}
Since the functional $V(W_1)H I_{L}\left(W_{1}\right)$ is $\mathbf{P}^{(m)}$-a.s. bounded and continuous, it follows from the quenched invariance principle (\ref{QIP for meander}) that the second term tends to zero.

Now we are left to consider the first term of right hand side of (\ref{triangle inequality 2}). Recall (\ref{asymptotic behaviour of harmonic function 2}), by the uniform convergence theorem of regularly varying functions, Theorem 1.5.2 in Bingham, Goldie and Teugels \cite{BGT87}, we have for any $0<L<\infty$, $$U(\theta^N\xi,\sqrt{N}\sigma x)\sim \sqrt{N}\sigma x, ~~~~as~~N\to\infty,$$ uniformly for $x\in (0,L]$. Therefore, by (\ref{rescaled harmonic function}), $$\tilde{U}(\xi,N,x)\sim \frac{\sqrt{N}\sigma x \mathbb{P}_\xi(\tau_{0}>N)}{U(\xi,0)},$$ uniformly for $x\in (0,L]$. It follows from (\ref{asymptotic behaviour of RW stay positive}) that
\begin{equation}\label{uniform convergence}
	\sup_{x\in(0,L]}\left|\tilde{U}(\xi,N,x)-\sqrt{\frac{2}{\pi}}V(x)\right| \to0, ~~~~as~~N\to\infty.
\end{equation}
This yields (\ref{first term of triangle inequality}), and hence conclude the proof of Theorem \ref{Th3} for $x=0$.
\qed


\subsection{The case $x>0$: Convergence of $\mathbf{P}_{\xi,~x}^{+,N}$}

In this section, we prove Theorem \ref{Th3} for the case $x>0$, that is, for almost all $\xi$, we have $\mathbf{P}_{\xi,~x}^{+,N}\Longrightarrow \mathbf{P}_x^+,~\text{as}~ N\to \infty.$ As before, we only need to prove that the sequence of $\mathbf{P}_{\xi,~x}^{+,N}$ converges in law to $\mathbf{P}_x^+$ on $\Omega_1$ for almost all $\xi$. By the same route as in Section 3.1, we prove that, for any bounded and continuous functional $H$: $\Omega_1\to \mathbb{R}$, one has $\mathbf{E}_{\xi,~x}^{+,N}(H)\to \mathbf{E}_x^+(H)$ as $N\to\infty$.

It follows from the definitions (\ref{Prob of RW stay positive}) and (\ref{rescaled prob on Omega}) that
\begin{equation}\nonumber
	\mathbf{P}^{+,N}_{\xi,~x}(A)=\frac{\mathbf{E}^{N}_{\xi,~x}\left(\tilde{U}(\xi,N,W_1)1_{\{\underline{W}_1>0\}}1_A\right)}{U(\xi,\sqrt{N}\sigma x)\mathbb{P}_{\xi}(\tau_{0}>N)/U(\xi,0)}.
\end{equation}
So, by the definition (\ref{Prob of Bessel process}), we get $$\left|\mathbf{E}_{\xi,~x}^{+,N}(H)-\mathbf{E}_x^+(H)\right| =\left|\frac{\mathbf{E}^{N}_{\xi,~x}\left(\tilde{U}(\xi,N,W_1)1_{\{\underline{W}_1>0\}}H\right)}{U(\xi,\sqrt{N}\sigma x)\mathbb{P}_{\xi}(\tau_{0}>N)/U(\xi,0)}-\frac{\mathbf{E}_x\left(\sqrt{\frac{2}{\pi}}V(W_1)1_{\{\underline{W}_1 >0\}}H\right)}{\sqrt{\frac{2}{\pi}}V(x)}\right|.$$
By (\ref{asymptotic behaviour of harmonic function 1}) and (\ref{asymptotic behaviour of RW stay positive}), we have
\begin{equation}\label{a.s. convergence}
	\frac{U(\xi,\sqrt{N}\sigma x)\mathbb{P}_{\xi}(\tau_{0}>N)}{U(\xi,0)}\to \sqrt{\frac{2}{\pi}}V(x),~~~~as~~N\to\infty.
\end{equation}

Now we only need to prove that
\begin{equation}\label{weak convergence of unconditioned law}
	\mathbf{E}^{N}_{\xi,~x}\left(\tilde{U}(\xi,N,W_1)1_{\{\underline{W}_1>0\}}H\right)\longrightarrow \mathbf{E}_{x}\left(\sqrt{\frac{2}{\pi}}V(W_1)1_{\{\underline{W}_1>0\}}H\right), ~~~~as~~N\to\infty.
\end{equation}

We proceed as in Section 3.1. Recall that $I_{L}(x)$ is the truncation function defined in (\ref{truncation function}).
By the triangle inequality, we obtain

\begin{equation}
	\begin{aligned}
		&\left| \mathbf{E}^{N}_{\xi,~x} \left(\tilde{U}(\xi,N,W_1)1_{\{\underline{W}_1>0\}}H\right)-\mathbf{E}_{x}\left(\sqrt{\frac{2}{\pi}}V(W_1)1_{\{\underline{W}_1>0\}}H\right) \right| \\
		\leq &\left| \mathbf{E}^{N}_{\xi,~x}\left(\tilde{U}(\xi,N,W_1)1_{\{\underline{W}_1>0\}}H I_{L}\left(W_{1}\right)\right)-\mathbf{E}_{x}\left(\sqrt{\frac{2}{\pi}}V(W_1)1_{\{\underline{W}_1>0\}}H I_{L}\left(W_{1}\right)\right)\right| \\
		&+\left|	\mathbf{E}^{N}_{\xi,~x}\left(\tilde{U}(\xi,N,W_1)1_{\{\underline{W}_1>0\}}H\left(1-I_{L}\left(W_{1}\right)\right)\right)\right| \nonumber\\
		&+\left|\mathbf{E}_{x}\left(\sqrt{\frac{2}{\pi}}V(W_1)1_{\{\underline{W}_1>0\}}H\left(1-I_{L}\left(W_{1}\right)\right)\right)\right| \\
		\leq &\left| \mathbf{E}^{N}_{\xi,~x}\left(\tilde{U}(\xi,N,W_1)1_{\{\underline{W}_1>0\}}H I_{L}\left(W_{1}\right)\right)-\mathbf{E}_{x}\left(\sqrt{\frac{2}{\pi}}V(W_1)1_{\{\underline{W}_1>0\}}H I_{L}\left(W_{1}\right)\right)\right| \\
		&+M\mathbf{E}^{N}_{\xi,~x}\left(\tilde{U}(\xi,N,W_1)1_{\{\underline{W}_1>0\}}\left(1-I_{L}\left(W_{1}\right)\right)\right)\\
		&+M\mathbf{E}_{x}\left(\sqrt{\frac{2}{\pi}}V(W_1)1_{\{\underline{W}_1>0\}}\left(1-I_{L}\left(W_{1}\right)\right)\right),
	\end{aligned}
\end{equation}
where the last inequality holds since $H$ is bounded (let again $|H|\leq M$) and the terms \\ $\tilde{U}(\xi,N,W_1)1_{\{\underline{W}_1>0\}}\left(1-I_{L}\left(W_{1}\right)\right)$ and $\sqrt{\frac{2}{\pi}}V(W_1)1_{\{\underline{W}_1>0\}}\left(1-I_{L}\left(W_{1}\right)\right)$ are non-negative. \\
Note that (recall (\ref{Prob of Bessel process}), (\ref{Prob of RW stay positive}) and (\ref{rescaled prob on Omega})) the following two equalities hold:
\begin{equation}\label{two equalities}
	\begin{aligned}
		&\mathbf{E}^{N}_{\xi,~x}\left(\tilde{U}(\xi,N,W_1)1_{\{\underline{W}_1>0\}}\right)=\frac{U(\xi,\sqrt{N}\sigma x)\mathbb{P}_{\xi}(\tau_{0}>N)}{U(\xi,0)}, \\ &\mathbf{E}_{x}\left(\sqrt{\frac{2}{\pi}}V(W_1)1_{\{\underline{W}_1>0\}}\right)=\sqrt{\frac{2}{\pi}}V(x).
	\end{aligned}	
\end{equation}
Then,
\begin{equation}\label{triangle inequality 3}
	\begin{aligned}
		&\left| \mathbf{E}^{N}_{\xi,~x} \left(\tilde{U}(\xi,N,W_1)1_{\{\underline{W}_1>0\}}H\right)-\mathbf{E}_{x}\left(\sqrt{\frac{2}{\pi}}V(W_1)1_{\{\underline{W}_1>0\}}H\right) \right| \\
		\leq &\left| \mathbf{E}^{N}_{\xi,~x}\left(\tilde{U}(\xi,N,W_1)1_{\{\underline{W}_1>0\}}H I_{L}\left(W_{1}\right)\right)-\mathbf{E}_{x}\left(\sqrt{\frac{2}{\pi}}V(W_1)1_{\{\underline{W}_1>0\}}H I_{L}\left(W_{1}\right)\right)\right| \\
		&+M\left[\frac{U(\xi,\sqrt{N}\sigma x)\mathbb{P}_{\xi}(\tau_{0}>N)}{U(\xi,0)}-\mathbf{E}^{N}_{\xi,~x}\left(\tilde{U}(\xi,N,W_1)1_{\{\underline{W}_1>0\}}I_{L}\left(W_{1}\right)\right)\right] \\
		&+M\left[\sqrt{\frac{2}{\pi}}V(x)-\mathbf{E}_{x}\left(\sqrt{\frac{2}{\pi}}V(W_1)1_{\{\underline{W}_1>0\}}I_{L}\left(W_{1}\right)\right)\right].
	\end{aligned}
\end{equation}

We claim that for all $L>0$, as $N \rightarrow \infty$,
\begin{equation}\label{first term of triangle inequality 3}
	\left| \mathbf{E}^{N}_{\xi,~x}\left(\tilde{U}(\xi,N,W_1)1_{\{\underline{W}_1>0\}}H I_{L}\left(W_{1}\right)\right)-\mathbf{E}_{x}\left(\sqrt{\frac{2}{\pi}}V(W_1)1_{\{\underline{W}_1>0\}}H I_{L}\left(W_{1}\right)\right)\right| \rightarrow 0.
\end{equation}
In fact, by the triangle inequality again, we get
\begin{equation}\label{triangle inequality 4}
	\begin{aligned}
		&\left| \mathbf{E}^{N}_{\xi,~x}\left(\tilde{U}(\xi,N,W_1)1_{\{\underline{W}_1>0\}}H I_{L}\left(W_{1}\right)\right)-\mathbf{E}_{x}\left(\sqrt{\frac{2}{\pi}}V(W_1)1_{\{\underline{W}_1>0\}}H I_{L}\left(W_{1}\right)\right)\right| \\
		\leq&\left| \mathbf{E}^{N}_{\xi,~x}\left(\sqrt{\frac{2}{\pi}}V(W_1)1_{\{\underline{W}_1>0\}}H I_{L}\left(W_{1}\right)\right)-\mathbf{E}_{x}\left(\sqrt{\frac{2}{\pi}}V(W_1)1_{\{\underline{W}_1>0\}}H I_{L}\left(W_{1}\right)\right)\right| \\
		&+M\sup_{x\in(0,L]}\left|\tilde{U}(\xi,N,x)-\sqrt{\frac{2}{\pi}}V(x)\right|.
	\end{aligned}
\end{equation}
Since the functional $V(W_1)1_{\{\underline{W}_1>0\}}HI_{L}\left(W_{1}\right)$ is $\mathbf{P}_{x}$-a.s. bounded and continuous, the first term in the right hand side of $\left(\ref{triangle inequality 4}\right)$ vanishes as $N\to\infty$ by the quenched invariance principle of non-conditioning (\ref{QIP for random walk without conditioning}), then (\ref{first term of triangle inequality 3}) follows from (\ref{uniform convergence}).

Now we return to $\left(\ref{triangle inequality 3}\right)$. By $\left(\ref{a.s. convergence}\right)$ and $\left(\ref{first term of triangle inequality 3}\right)$ (let $H \equiv 1$), the second term in the right hand side of $\left(\ref{triangle inequality 3}\right)$ converges to the third term, and let $L\to\infty$, the third term tends to zero by (\ref{two equalities}), Therefore, the term in the left hand side of $\left(\ref{triangle inequality 3}\right)$ tends to zero as $N\to\infty$, (\ref{weak convergence of unconditioned law}) follows. This completes the proof of Theorem \ref{Th3}.
\qed


\section*{Acknowledgments}
\thanks{This work was supported in part by the National Key Research and Development Program of China (No.~2020YFA0712900), NSFC (No.~11971062).}

%


\end{document}